\documentclass[12pt]{article}

\usepackage{overpic,graphicx}
\usepackage{amssymb,amsfonts,amsmath,amsthm}
\usepackage{animate}
\usepackage{color}
\usepackage[english]{babel}
\usepackage[latin1]{inputenc}
\usepackage{times}
\usepackage{geometry}
\usepackage[breaklinks,pdfstartview=FitH,colorlinks,linktocpage,bookmarks=true]{hyperref}

\newtheorem{thm}{Theorem}[section]

\newtheorem{lem}[thm]{Lemma}
\newtheorem{rem}[thm]{Remark}
\newtheorem{exmp}[thm]{Example}

\newtheorem{conj}[thm]{Conjecture}
\newtheorem{cor}[thm]{Corollary}

\newtheorem{lemma}[thm]{Lemma}

\newtheorem{mainthm}[thm]{Main Conjecture}
\newtheorem{splitting-lemma}[thm]{Splitting Lemma}
\def\re{\mathrm{Re}}
\def\im{\mathrm{Im}}

\def\bar#1{\overline{#1}}

\def\deg{\mathop{\mathrm{deg}}\nolimits}

\def\R{{\mathbb R}}

\def\H{{\mathbb H}}

\long\def\comment#1\endcomment{}

\begin{document}


\title{Surfaces containing two circles through each point \\ and Pythagorean 6-tuples}

\author{M. Skopenkov, R. Krasauskas}
\date{}

\maketitle

\begin{abstract}
We study analytic surfaces in 3-dimensional Euclidean space containing two circular arcs through each point. The problem of finding such surfaces traces back to the works of Darboux 
from XIXth century. We reduce finding all such surfaces to the algebraic problem of finding all Pythagorean 6-tuples of polynomials.
The reduction is based on the Schicho parametrization of surfaces containing two conics through each point and a new approach using quaternionic rational parametrization.

\smallskip


\end{abstract}

\footnotetext[0]{
The article was prepared within the framework of the Academic Fund Program at the National Research University Higher School of Economics (HSE) in 2015-2016 (grant No 15-01-0092) and supported within the framework of a subsidy granted to the HSE by the Government of the Russian Federation for the implementation of the Global Competitiveness Program.
During the work on this paper the first author received support 
also from ``Dynasty'' foundation and from the Simons--IUM fellowship.
The second author was partially supported by the Marie-Curie Initial Training Network SAGA, FP7-PEOPLE contract PITN-GA-214584.
}

\section{Introduction}

We study surfaces in space $\mathbb{R}^3$ 
such that through each point of the surface one can draw two circles fully contained in the surface. Hereafter by a \emph{circle} we mean either an ordinary circle in $\mathbb{R}^3$ or a straight line.
In this paper we 
reduce finding all such surfaces to the algebraic problem of finding all Pythagorean $6$-tuples of polynomials.  In a subsequent publication we are going to solve the latter problem.

The problem of finding such surfaces traces back to the works of Darboux 
from XIXth century. Basic examples --- a one-sheeted hyperboloid and a nonrotational ellipsoid --- are discussed in Hilbert--Cohn-Vossen's ``Anschauliche Geometrie''. There (and respectively, in a recent paper \cite{NS11} by Nilov and the first author) it is also proved  that a smooth surface containing two lines (respectively, a line and a circle) through each point is a quadric or a plane. A torus contains $4$ circles through each point: a meridian, a parallel, and two Villarceau circles.

\begin{figure}[bht]
\begin{center}
\begin{tabular}{c}
\includegraphics[height=2.2cm]{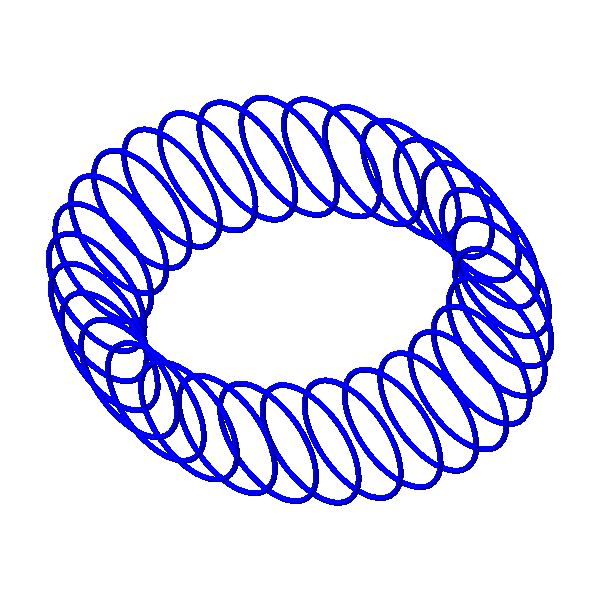}
\end{tabular}
\hspace{-0.5cm}
\begin{tabular}{c}
\includegraphics[height=2.2cm]{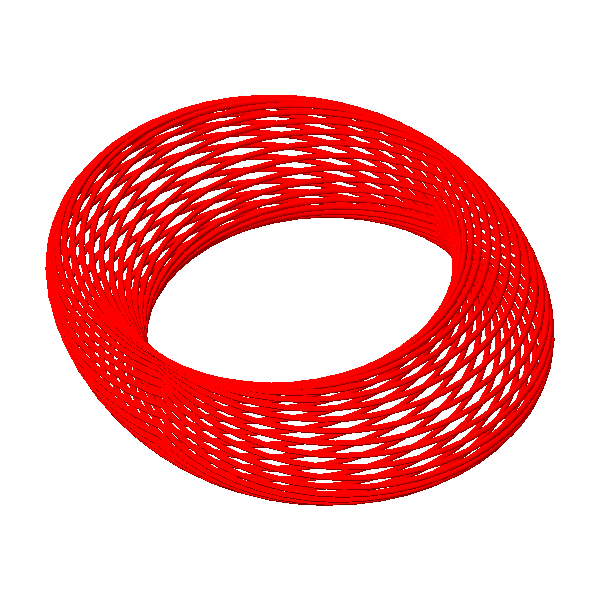}
\end{tabular}
\begin{tabular}{c}
\includegraphics[height=1.8cm]{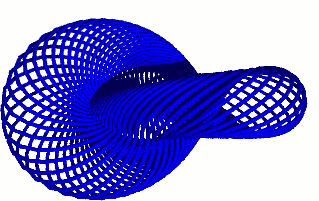}
\end{tabular}
\hspace{-0.5cm}
\begin{tabular}{c}
\includegraphics[height=1.8cm]{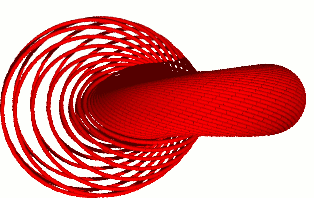}
\end{tabular}
\begin{tabular}{c}
\includegraphics[width=0.15\textwidth]{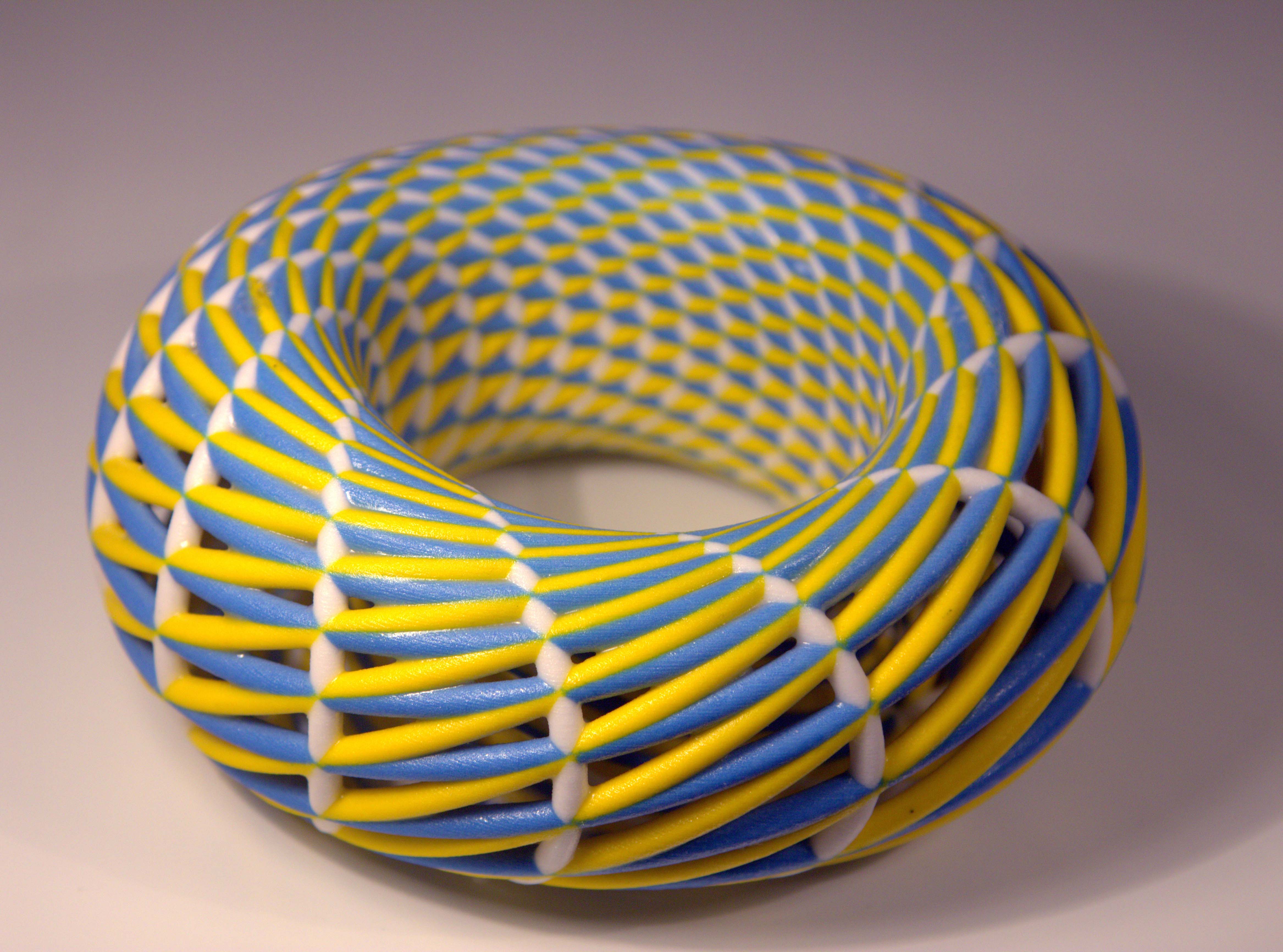}
\end{tabular}	
\end{center}
\caption{A Darboux cyclide, 
Euclidean and Clifford translational surfaces \cite{Niels-13}.}
\label{movie}
\end{figure}



All these examples are particular cases of a \emph{Darboux cyclide}, i.e., a subset of $\mathbb{R}^3$ given by the equation
$$
a(x^2+y^2+z^2)^2+(x^2+y^2+z^2)(bx+cy+dz)+Q(x,y,z)=0,
$$
where $a,b,c,d\in\mathbb{R}$ and $Q\in\mathbb{R}[x,y,z]$ of degree $\le 2$ do not vanish simultaneously; see Figure~\ref{movie} to the left.
Equivalently, a \emph{Darboux cyclide} is the stereographic projection of the intersection of the sphere $S^3$ with another $3$-dimensional quadric \cite[Section~2.2]{PSS11}. Almost each Darboux cyclide contains at least $2$ circles through each point (and there is an effective algorithm to count their actual number \cite{PSS11,Takeuchi-00}). Conversely, Darboux has shown that 
$10$ circles through each point guarantee that an analytic 
surface is a Darboux cyclide. 
This result has been improved over the years: in fact already $3$, or $2$ orthogonal, or $2$ cospheric circles are sufficient for the same conclusion
\cite[Theorem~3]{Niels-13},
\cite[Theorem~1]{ivey:1995},
\cite[Theorem~20 in p.~296]{coolidge:1916},
  cf. \cite{Blum,
NS11}.
Hereafter two circles are called \emph{cospheric}, if they are contained in one $2$-dimensional sphere or plane.

Recently there has been a renewed interest to surfaces containing $2$ circles through each point due to Pottmann who considered their potential applications to architecture \cite{PSS11}.
Pottmann noticed that the \emph{Euclidean translational surface} 
$\{\,p+q:p\in\alpha,q\in\beta\,\}$,  where $\alpha,\beta$ are two fixed generic circles in $\mathbb{R}^3$,
contains $2$ circles through each point but is not a Darboux cyclide 
\cite[Example~3.9]{NS11}.
Another example with similar properties was given by Zub\.e in 2011: the stereographic projection of a \emph{Clifford translational surface} 
$\{\,p\cdot q:p\in\alpha,q\in\beta\,\}$, where $\alpha,\beta$ are now circles in the sphere ${S}^3$ identified with the set of unit quaternions. 
The projection itself is called a \emph{Clifford translational surface} as well; see Figure~\ref{movie} to the right. It may have degree up to $8$. Each degree $8$ surface in $S^3$ containing a \emph{great} circle and another circle through each point is Clifford translational \cite[Corollary~2c]{Lubbes-15}.
More examples can be obtained from these translational surfaces by \emph{M\"obius transformations}, i.e., compositions of inversions. 
Euclidean and Clifford translational surfaces are not M\"obius transformations of each other \cite[Theorem 2b]{Lubbes-15}.
For related transformations taking lines to circles see \cite{Timorin-06}.

We conjecture that the above ones are the only possible surfaces containing $2$ circles through each point. Let us make this statement precise. We switch to a local problem involving a piece of a surface instead of a closed one and circular arcs instead of circles. By an \emph{analytic surface} in $\mathbb{R}^n$ we mean the image of an injective real analytic map of a planar domain into $\mathbb{R}^n$ with nondegenerate differential at each point. We use the same notation for the map and the surface; no confusion arises from this. A circle (or circular arc) \emph{analytically depending} on a point is a real analytic map of an analytic surface into the 
variety of all circles (or circular arcs) in $\mathbb{R}^n$. 
Analyticity is not really a restriction \cite{Kataoka-Takeuchi-13}.


\begin{mainthm}\label{mainthm}
If through each point of an analytic surface in $\mathbb{R}^3$ one can draw two transversal circular arcs fully contained in the surface (and analytically depending on the point) then the surface is a M\"obius transformation of a subset of either a Darboux cyclide, or Euclidean or Clifford translational surface.
\end{mainthm}


We hope to deduce Main Conjecture~\ref{mainthm} from the following 4-dimensional counterpart. The $4$-dimensional problem seems to be more accessible than the $3$-dimensional one because of nice approach using quaternions. In what follows identify $\mathbb{R}^4$ with the skew field $\mathbb{H}$ of quaternions, and $\mathbb{R}^3$ with the set $\mathrm{Im}\mathbb{H}$ of purely imaginary quaternions. M\"obius transformations in $\mathbb{R}^4$ are precisely the nondegenerate maps of the form $q\mapsto (aq+b)(cq+d)^{-1}$ and $q\mapsto (a\bar q+b)(c\bar q+d)^{-1}$, where $a,b,c,d\in\mathbb{H}$;
see~\cite{Lavicka-etal-07} for an exposition. Circles in $\mathbb{R}^4$ are precisely the nondegenerate curves having a parametrization of the form $\alpha(u)=(au+b)(cu+d)^{-1}$ (outside one point),
where $a,b,c,d\in\mathbb{H}$ are fixed and $u\in\mathbb{R}$ runs. Denote by $\H_{mn}\subset\mathbb{H}[u,v]$ the set of polynomials with quaternionic coefficients of degree at most $m$ in the variable $u$ and at most $n$ in the variable $v$ (the variables commute with each other and the coefficients). Denote $\H_{m*}:=\bigcup_{n=1}^{\infty}\H_{mn}$. Define $\H_{**}$, $\mathbb{C}_{mn}$, and $\mathbb{R}_{mn}$ analogously. For each $P\in\mathbb{H}_{mn}$ and real numbers $\hat u,\hat v$ (but not quaternions) the value $P(\hat u,\hat v)$ is well-defined. Thus the polynomial $P$ or a rational expression in such polynomials defines a surface in $\mathbb{R}^4$.





\begin{conj}
\label{4DBconj}
Assume that through each point of an analytic surface in $\mathbb{R}^4$ one can draw two noncospheric circular arcs fully contained in the surface (and analytically depending on the point).
Assume that for each point in some dense subset of the surface the number of circular arcs passing through the point and fully contained in the surface is finite. Then some M\"obius transformation of the surface has a parametrization
\begin{equation}\label{eq-4DBconj}
\Phi(u,v)=A(u,v)^{-1}B(u,v)C(u,v)^{-1}
\end{equation}
for some polynomials $A,B,C\in\mathbb{H}_{11}$ such that $AC\in \mathbb{H}_{11}$.
\end{conj}

Conversely, 
almost each surface~\eqref{eq-4DBconj} contains two circular arcs $u=\mathrm{const}$ and $v=\mathrm{const}$ through each point because the curves $\alpha(u)=(au+b)(cu+d)^{-1}$ and $\beta(u)=(cu+d)^{-1}(au+b)$ are circular arcs  for almost each $a,b,c,d\in\mathbb{H}$. E.g., $\Phi(u,v)=(v+i)^{-1}(v+j)(u+k)(u+i)^{-1}$ is a Clifford translational surface, and $\Phi(u,v)=(u-i)((2j+i)v-2i-j)\left((u-i)(v-k)\right)^{-1}$ is a torus.

The first result of this paper is the following assertions reducing the
$3$-dimensional problem to the $4$-dimensional one. They are proved in Section~\ref{sec:proofs}.

\begin{thm}\label{prop21}
\label{prop:21-split}
If surface~\eqref{eq-4DBconj} 
is contained in $\mathbb{R}^3$ (respectively, in $S^3$) then it is a subset of either an
Euclidean (respectively, Clifford) translational surface or
a Darboux cyclide (respectively, an intersection of $S^3$ with another $3$-dimensional quadric).
\end{thm}

\begin{cor}\label{cor-3D} If a surface in $\mathbb{R}^3$ is a M\"obius transformation of surface~\eqref{eq-4DBconj} in $\mathbb{R}^4$ then the former surface is a M\"obius transformation of a subset of either a Darboux cyclide, or Euclidean or Clifford translational surface.
\end{cor}

Surfaces containing two circles through each point are particular cases of surfaces containing two conic sections or lines through each point. The latter surfaces have been classified by Brauner, Schicho, and Lubbes \cite{Brauner, schicho:2001, Lubbes-14}. In the particular case of so-called supercyclides 
a classification was given by Degen \cite{Degen-86}. The Schicho classification is up to a week equivalence relation, thus it does not allow automatically to find all surfaces containing two circles through each point. However the following parametrization result by Schicho provides the first step toward the solution of our problem. The notions used in the statement are defined analogously to the above; see Section~\ref{sec:proofs2} for details.

\begin{thm}\label{schicho} 
Assume that through each point of
an analytic surface in 
a domain in $n$-dimensional complex projective space
one can draw two transversal conic sections intersecting each other only at this point (and analytically depending on the point) such that their intersections with the domain are contained in the surface.
Assume that through each point in some dense subset of the surface one can draw only finitely many conic sections such that their intersections with the domain are contained in the surface. Then
the surface (possibly besides a one-dimensional subset) has a parametrization
\begin{equation}\label{eq-circular}
\Phi(u,v)=X_{0}(u,v):\dots:X_{n}(u,v)
\end{equation}
for some $X_{0},\dots,X_{n}\in\mathbb{C}_{22}$
such that
the conic sections are the curves $u=\mathrm{const}$ and $v=\mathrm{const}$.
%
\end{thm}

Conversely, one can see immediately that almost each surface~\eqref{eq-circular} contains two conic sections $u=\mathrm{const}$ and $v=\mathrm{const}$ through each point.
Theorem~\ref{schicho} in particular implies that each surface containing two conic sections through each point is contained in an projective subspace of dimension at most $8$ and has the degree at most $8$ (by standard elimination of variables).

Theorem~\ref{schicho} is proved completely analogously to \cite[Theorem~11]{schicho:2001}, where the case when $n=3$ is considered; see also \cite[Theorem~17]{Lubbes-14}. For convenience of the reader we give the proof in Section~\ref{sec:proofs2}. In comparison to \cite{schicho:2001} we have added some details 
to make it accessible to nonspecialists.

The main result of the paper is the following corollary also proved in Section~\ref{sec:proofs2}.

\begin{cor}\label{haupt}\label{param} 
Assume that through each point of
an analytic surface in $S^{n-1}$ (respectively, in $\mathbb{R}^n$)
one can draw two noncospheric circular arcs fully contained in the surface (and analytically depending on the point).
Assume that through each point in some dense subset of the surface one can draw only finitely many circular arcs fully contained in the surface.
Then the surface (possibly besides a one-dimensional subset) has a parametrization
\begin{equation*}
\Phi(u,v)=X_{0}(u,v):\dots:X_{n}(u,v),
\end{equation*}
where $X_{0},\dots,X_{n}\in\mathbb{R}_{22}$
satisfy the equation
\begin{equation}\label{eq-sphere}
X_{1}^2+\dots+X_{n}^2=X_{0}^2
\end{equation}
(respectively, the equation
$X_{1}^2+\dots+X_{n}^2=X_{0}Y$
for some $Y\in\mathbb{R}_{22}$). 
\end{cor}

This result reduces the above conjectures to 
solving the equation $X_1^2+\dots+X_{n-1}^2=X_{n}^2$ for ``Pythagorean $n$-tuples'' in real polynomials. 
The resulting problem is hard but seems more accessible because of possible induction over the involved parameters ($n$, number of variables, degree).
The problem has been solved for $n=3$ and $4$ in \cite[Theorem~2.2]{Dietz-etal} using
that $\mathbb{C}[u,v]$ is a unique factorization domain (UFD).
In case of one variable a similar argument works for $n=5$ and $6$ because $\mathbb{H}[u]$ is still a UFD in a sense \cite[Theorem~1 in Chapter~2]{Ore}, cf. \cite{Gelfand-etal-05, Motzkin-etal}, \cite[\S3.5]{Gentili-etal-13} and Conjecture~\ref{conj-1-var} below. The main difficulty of passing to two variables is that $\mathbb{H}[u,v]$ is \emph{not} a UFD any more; see Example~\ref{ex-Beauregard} below taken from~\cite{Beauregard-93}. 
The case of two variables and $n=6$ arising in our geometric problem seems to be the simplest case not accessible by the methods available before.


%
%


So far we give only one result in this direction, which is nice and interesting in itself, useful for the proof of the above conjectures and also for Theorem~\ref{prop21} above. The result is proved in Section~\ref{sec:final}.

\begin{splitting-lemma}\label{l-splitting-basic}
If $|Q(u,v)|^2=P(u)R(v)$ for some $Q\in\mathbb{H}_{11}$ and $P\in\mathbb{R}_{20}$, $R\in\mathbb{R}_{02}$ of degree~$2$
then $Q$ is reducible.
\end{splitting-lemma}

\begin{cor}\label{cor-splitting}
If $X_{1}^2+\dots+X_5^2=X_6^2$ for some $X_{1},\dots,X_4\in\mathbb{R}_{11}$,  $X_5,X_6\in\mathbb{R}_{22}$ such that $X_5+ X_6\in\mathbb{R}_{20}$ and $X_5-X_6\in \mathbb{R}_{02}$ have degree $2$ then for some $A\in\mathbb{H}_{10}$ and $B\in\mathbb{H}_{01}$ we have
$$X_1+iX_2+jX_3+kX_4=AB \quad\text{ or } \quad X_1+iX_2+jX_3+kX_4=BA.$$
\end{cor}

To summarize, the results of this paper reduce Main Conjecture~\ref{mainthm} to the following completely algebraic conjecture. 
By a \emph{M\"obius transformation} of $S^4$ we mean a linear transformation $\mathbb{R}^6\to\mathbb{R}^6$ which preserves the homogeneous equation $x_{1}^2+\dots+x_5^2=x_6^2$ of $S^4$.

\begin{conj}\label{conj-alg} 
Polynomials $X_{1},\dots,X_6\in\mathbb{R}_{22}$ satisfy equation $X_{1}^2+\dots+X_5^2=X_6^2$ if and only if
up to M\"obius transformation of $S^4$ \textup{(}not depending on $u,v$\textup{)} we have 
\begin{equation}\label{eq-param-pythagorean}
\begin{aligned}
X_1+iX_2+jX_3+kX_4&=2ABCD,\\
X_5&=(|B|^2-|AC|^2)D,\\
X_6&=(|B|^2+|AC|^2)D
\end{aligned}
\end{equation}
for some $A,B,C\in\mathbb{H}_{11}$, $D\in\mathbb{R}_{22}$ such that $|B|^2D,|AC|^2D\in\mathbb{R}_{22}$.
\end{conj}

The parametrizations of Conjectures~\ref{conj-alg} and~\ref{4DBconj} are related to each other via the stereographic projection $X_1:\dots:X_6\mapsto (X_1,X_2,X_3,X_4)/(X_6-X_5)$. 

We plan to prove all the above conjectures in a subsequent publication. This is reasonable because the remaining algebraic part of the conjectures is proved by very different methods.


\section{Classification results}\label{sec:proofs}

In this section we prove the results on the classification of surfaces parametrized by quaternionic rational functions of small degree up to M\"obius transformation: Theorem~\ref{prop21} and Corollary~\ref{cor-3D}.

Let us prove several lemmas required for the proof of Theorem~\ref{prop21}. These lemmas are independent in the sense that the proof of each one uses the statements but not the proofs of the other ones.

In what follows $A,B\in\mathbb{H}_{11}$ are arbitrary polynomials not vanishing identically. Linear homogeneous polynomials $\tilde A,\tilde B\in\mathbb{H}[u,v,w,s]$ are defined by the formulae $A(u,v)=\tilde A(u,v,uv,1)$ and $B(u,v)=\tilde B(u,v,uv,1)$. By a \emph{possibly degenerate surface} we mean
an analytic map of a planar domain into $\mathbb{R}^n$ not necessarily with nondegenerate differential, 
and also the image of the map, 
if no confusion arises. A \emph{possibly degenerate hypersurface} is defined analogously.



First we consider the (possibly degenerate) surface $\Phi(u,v)=A(u,v)B(u,v)^{-1}$ being a particular case of surface~\eqref{eq-4DBconj}. We are going to estimate its degree.
For that we estimate the degree of the (possibly degenerate) hypersurface $$\tilde\Phi(u,v,w,s)=\tilde A(u,v,w,s)\tilde B(u,v,w,s)^{-1}.$$
It has the rational parametrization $\tilde \Phi=\tilde A\overline{\tilde B}/|\tilde B|^2$ of degree at most $2$ in each variable, and by elimination of variables (not used in the paper)
it has degree at most $8$. We prove a much sharper estimate.

\begin{lem}\label{l-3folddeg4} 
The (possibly degenerate) hypersurface $\tilde A\tilde B^{-1}$
is contained in an algebraic hypersurface of degree at most $4$.
\end{lem}

\begin{proof} If a point $t+ix+jy+kz\in\mathbb{H}$ is contained in the hypersurface $\tilde A\tilde B^{-1}$ then $\tilde A(u,v,w,s)-(t+ix+jy+kz)\cdot\tilde B(u,v,w,s)=0$ for some $u,v,w,s\in\mathbb{R}$ not vanishing simultaneously. The latter quaternionic equation can be considered as a system of $4$ real linear homogeneous equations in the variables $u,v,w,s$ with the coefficients linearly depending on the parameters $x,y,z,t$. The system has a nonzero solution if and only if the determinant of the system vanishes, which gives an algebraic equation in $x,y,z,t$ of degree at most $4$. The algebraic equation defines the required algebraic hypersurface unless the determinant vanishes identically.

Assume that the determinant vanishes identically. Then the system has a nonzero solution for each point $(x,y,z,t)\in\mathbb{R}^{4}$. But the set of values of the fraction $\tilde A\tilde B^{-1}$ is at most three-dimensional. One can have a nonzero solution for each point in $\mathbb{R}^{4}$ only if $\tilde A(\hat u,\hat v,\hat w,\hat s)=\tilde B(\hat u,\hat v,\hat w,\hat s)=0$ for some nonzero $(\hat u,\hat v,\hat w,\hat s)\in\mathbb{R}^{4}$. Assume that $\hat s\ne 0$ without loss of generality. Then by the linearity
\begin{align*}
\tilde A(u,v,w,s)&=\tilde A(\hat s u-\hat us, \hat s v-\hat vs, \hat s w-\hat ws,0)/\hat s\\
\tilde B(u,v,w,s)&=\tilde B(\hat s u-\hat us, \hat s v-\hat vs, \hat s w-\hat ws,0)/\hat s.
\end{align*}
Performing a linear change of the parameters $u,v,w,s$ we may assume that both $\tilde A(u,v,w,s)$ and $\tilde B(u,v,w,s)$ do not depend on $s$. Further denote by $\tilde A(u,v,w)$ and $\tilde B(u,v,w)$ the resulting linear polynomials, defining the same hyperfurface $\tilde A\tilde B^{-1}$ as the initial ones.

Then the above equation takes the form $\tilde A(u,v,w)-(t+ix+jy+kz)\cdot\tilde B(u,v,w)=0$. Consider the equation as a system of $4$ real linear homogeneous equations in the variables $u,v,w$. The system has a nonzero solution if and only if all the $3\times 3$ minors of the system vanish, which gives four algebraic equations in $x,y,z,t$ of degree at most $3$. If at least one of the $3\times 3$ minors does not vanish identically then it defines
the required algebraic surface. If all the $3\times 3$ minors vanish identically then repeat the argument of the previous paragraph to get
linear homogeneous $\tilde A(u,v)$ and $\tilde B(u,v)$ depending only on $2$ variables. Again, we either get the required algebraic surface or proceed to the case when $\tilde A(u)$ and $\tilde B(u)$ depend only on $1$ variable. In the latter case any hyperplane passing through the point $\tilde A(1)\tilde B(1)^{-1}$ is the required algebraic hypersurface.
\end{proof}

We consider the case when the constructed algebraic hypersurface degenerates to $\mathrm{Im}\mathbb{H}$ separately.

\begin{lem}\label{l-3fold-R3} 
Assume that the (possibly degenerate) hypersurface $\tilde A\tilde B^{-1}$ is contained in the hyperplane $\mathrm{Im}\mathbb{H}$. Then the map $\tilde A\tilde B^{-1}$ is a composition of a map of the form either $\tilde C\tilde D\tilde C^{-1}$ or $\tilde E\tilde F^{-1}$ with a M\"obius transformation of $\mathrm{Im}\mathbb{H}$ (with constant coefficients), where $\tilde C\in\mathbb{H}[u,v,w,s]$, $\tilde E\in\mathrm{Im}\mathbb{H}[u,v,w,s]$, $\tilde F\in\mathbb{R}[u,v,w,s]$ are linear homogeneous, and $\tilde D\in\mathrm{Im}\mathbb{H}$.
\end{lem}

\begin{proof} The inverse stereographic projection
$
\im\H \to S^3, q \mapsto (q+1)(q-1)^{-1},
$
maps the hypersurface $\tilde A\tilde B^{-1}$ contained in $\im\mathbb{H}$ to the hypersurface $(\tilde A+\tilde B)(\tilde A-\tilde B)^{-1}$ contained in $S^3$. Thus $|\tilde A+\tilde B|=|\tilde A-\tilde B|$ identically. In particular, for each $(u,v,w,s)\in\mathbb{R}^4$ the condition $\tilde A(u,v,w,s)-\tilde B(u,v,w,s)=0$ implies the condition $\tilde A(u,v,w,s)+\tilde B (u,v,w,s)=0$.

Denote $\tilde A(u,v,w,s)-\tilde B(u,v,w,s)=:a_1u+a_2v+a_3w+a_4s$
and $\tilde A(u,v,w,s)+\tilde B(u,v,w,s)=:b_1u+b_2v+b_3w+b_4s$. Define a real linear map from the linear span of $a_1,a_2,a_3,a_4\in\mathbb{H}$ into the linear span of $b_1,b_2,b_3,b_4\in\mathbb{H}$ by the formula
$a_1u+a_2v+a_3w+a_4s\mapsto b_1u+b_2v+b_3w+b_4s$. The map is well-defined because the condition $a_1u+a_2v+a_3w+a_4s=0$ implies $b_1u+b_2v+b_3w+b_4s=0$. The map is isometric because $|\tilde A+\tilde B|=|\tilde A-\tilde B|$. Extend it to an isometry $\mathbb{H}\to\mathbb{H}$. Each isometry $\mathbb{H}\to\mathbb{H}$ has one of the forms $q\mapsto cqd$ or $q\mapsto c\bar qd$ for some $c,d\in S^3$ \cite[Theorem~3.2]{Lavicka-etal-07}. Therefore either $\tilde A+\tilde B=c(\tilde A-\tilde B)d$ or $\tilde A+\tilde B=c\overline{(\tilde A-\tilde B)}d$.

In the former case set $\tilde C:=\tilde A-\tilde B$ and $\tilde D:=(d+1)(d-1)^{-1}$. We have $(\tilde A+\tilde B)(\tilde A-\tilde B)^{-1}=c \tilde C d \tilde C^{-1}$. Up to a M\"obius transformation this is $\tilde C d \tilde C^{-1}$, which projects stereographically to $\tilde C \tilde D \tilde C^{-1}$.

In the latter case set $\tilde E:=\mathrm{Im}(\tilde B\bar d-\tilde A\bar d)$ and $\tilde F:=\mathrm{Re}(\tilde A\bar d-\tilde B\bar d)$. We have $(\tilde A+\tilde B)(\tilde A-\tilde B)^{-1}=-c \bar d (\tilde E+\tilde F)(\tilde E-\tilde F)^{-1}$. Up to a M\"obius transformation this is $(\tilde E+\tilde F)(\tilde E-\tilde F)^{-1}$, which projects stereographically to $\tilde E \tilde F^{-1}$.
\end{proof}

\begin{lem}\label{l-quadrics} The (possibly degenerate) surfaces $C D C^{-1}$ and $ E F^{-1}$, where $ C\in\mathbb{H}_{11}$, $E\in\mathrm{Im}\mathbb{H}_{11}$, $F\in\mathbb{R}_{11}$, and $D\in\mathrm{Im}\mathbb{H}$, are contained in certain quadrics in $\mathrm{Im}\mathbb{H}$.
\end{lem}

\begin{proof} The  surface $CD C^{-1}$ is a subset of the sphere $\{q\in\mathrm{Im}\mathbb{H}:|q|=|D|\}$. The  surface $ E F^{-1}$ is a rational surface in $\mathbb{R}^3$ of degree at most $1$ in each variable, hence a subset of a quadric.
\end{proof}

\begin{lem}\label{l-deg4} 
If the  (possibly degenerate) surface $AB^{-1}$ is contained in $\mathrm{Im}\mathbb{H}$ then it is contained in an irreducible algebraic surface of degree at most $4$.
\end{lem}

\begin{proof} By Lemma~\ref{l-3folddeg4} the hypersurface $\tilde A\tilde B^{-1}$, and hence $AB^{-1}$, is contained in an irreducible algebraic hypersurface of degree $\le 4$. One of the components of the intersection of the algebraic hypersurface with the hyperplane $\mathrm{Im}\mathbb{H}$ is the required surface unless the algebraic hypersurface coincides with the hyperplane.
In the latter case the assumptions of Lemma~\ref{l-3fold-R3} are satisfied. By Lemmas~\ref{l-3fold-R3} and~\ref{l-quadrics} the surface $AB^{-1}$ is a M\"obius transformation of a quadric, hence has degree $\le 4$.
\end{proof}

Now the following folklore lemma allows us to prove a particular case of Theorem~\ref{prop21}. 
The lemma is a particular case of \cite[Theorem~11]{Krasauskas-etal-13}. 
We give the proof for convenience of the reader.

\begin{lem}\label{l-cyclide} 
If two irreducible algebraic surfaces of degree at most $4$ in $\mathbb{R}^3$ are symmetric with respect to the unit sphere then they both are Darboux cyclides.
\end{lem}

\begin{proof} Denote $G(x,y,z):=x^2+y^2+z^2$. Let $A(x,y,z)=0$ and $B(x,y,z)=0$ be equations of our surfaces of minimal degrees. Since the surfaces are symmetric with respect to the unit sphere it follows that $B(x,y,z)=c G(x,y,z)^k A(\frac{x}{G(x,y,z)},\frac{y}{G(x,y,z)},\frac{z}{G(x,y,z)})$ for some $0\le k\le 4$ and $c\ne 0$.

Expand $A=A_4+A_3+A_2+A_1+A_0$ and $B=B_4+B_3+B_2+B_1+B_0$, where $A_i$ and $B_i$ are homogeneous of degree $i$.
We get $B=cG^k(A_0+G^{-1}A_1+\dots +G^{-4}A_4)$.
Hence $G^{4-k}B=c(G^4A_0+G^{3}A_1+\dots +A_4)$.

Assume that $A_0,B_4\ne 0$ (otherwise the proof is analogous).
The degree of the left- and right-hand sides equal $12-2k$ and $8$ respectively. Thus $k=2$. Comparing the highest-degree terms  we get $G^2B_4=cG^4A_0$, hence $B_4$ is divisible by $G^2$. Comparing the degree $7$ terms we get $G^2B_3=cG^3A_1$, hence $B_3$ is divisible by $G$. Therefore $B(x,y,z)=0$ is a Darboux cyclide. Analogously $A(x,y,z)=0$ is a Darboux cyclide.
\end{proof}


\begin{lem}\label{prop11}
If the surface  $\Phi(u,v)=A(u,v)B(u,v)^{-1}$, where $A,B\in\mathbb{H}_{11}$, is contained in $\mathbb{R}^3$ (respectively, in $S^3$) then it is a subset of a Darboux cyclide (respectively, an intersection of $S^3$ with another $3$-dimensional quadric).
\end{lem}

\begin{proof}
For a surface in $ \R^3$ the result follows from Lemmas~\ref{l-deg4} and~\ref{l-cyclide} applied to the surfaces $AB^{-1}$ and $-BA^{-1}$.
Now consider a surface in $S^3$. Project it stereographically to $\mathbb{R}^3$.
The resulting surface has the form $(A+B)(A-B)^{-1}$. By the lemma for $\mathbb{R}^3$ it is a Darboux cyclide.  By \cite[Section~2.2]{PSS11} the initial surface is an intersection of $S^3$ with another quadric.
\end{proof}

Let us proceed to the general case of Theorem~\ref{prop21}. We start with a folklore lemma.

\begin{lemma}\label{l-axial} The curve $\gamma(u)=(au+b)(cu+d)^{-1}\subset\mathbb{H}$, where $a,b,c,d\in\mathbb{H}$, $c\ne 0$, $b-ac^{-1}d\ne 0$, are fixed and $u\in\mathbb{R}$ runs, is a circle (possibly without one point). If $\gamma(u)$ is contained in $\im\H$ then the plane of the circle is orthogonal to the vector $\im\, (dc^{-1})$.
\end{lemma}

\begin{proof} We have
$\gamma(u)=ac^{-1}+(b-ac^{-1}d)(cu+d)^{-1}=:f+g(u+h)^{-1}$ for some $f,g,h\in\mathbb{H}$, where $h=dc^{-1}$.
This is a composition of translations, rotations, and an inversion applied to the line $\mathbb{R}\subset \mathbb{H}$, once $c\ne 0$ and $b-ac^{-1}d\ne 0$. Thus $\gamma(u)$ is a circle (possibly without one point).

Now assume that $\gamma(u)\subset\im\H$. Then $0=\re\gamma(u)|u+h|^2=\re f|u+h|^2+\re g u+\re (g\bar h)$. Since the coefficients of the polynomial in the right-hand side vanish it follows that
$f,g\in\im\H$ and
$0=\mathrm{Re}(g\bar h)=\mathrm{Re}(g\mathrm{Re}h+g\cdot\mathrm{Im}h-g\times\mathrm{Im}h)=g\cdot\mathrm{Im}h$, hence
$g$ is orthogonal to $\im h$. Assume further that $\im h\ne 0$, otherwise $\gamma(u)$ is a line and there is nothing to prove. Then the circle $g(u+h)^{-1}=(gu+g\mathrm{Re}h-g\times\mathrm{Im}h)|u+h|^{-2}$ is contained in the plane spaned by $g$ and $g\times \im h$. So $\gamma(u)=f+g(u+h)^{-1}$ is contained in a plane orthogonal to $\im h$ as well.
\end{proof}

A similar result holds for the curve $\gamma(u)=(cu+d)^{-1}(au+b)$, where $b-dc^{-1}a\ne0$.



\begin{proof}[Proof of Theorem~\ref{prop21} for $\Phi\subset\mathbb{R}^3$] 
If $A=\mathrm{const}$ then 
by Lemma~\ref{prop11} $\Phi=(A^{-1}B)C^{-1}$ is a subset of a Darboux cyclide. 
Analogously, if $C=\mathrm{const}$ then
$\Phi=-\overline{\Phi}=(-\overline{C}^{-1}\overline{B})(\overline{A})^{-1}$ 
is a Darboux cyclide.

Assume further that $A,C\ne\mathrm{const}$. Since $AC\in\mathbb{H}_{11}$ it follows that either $A\in\mathbb{H}_{01}$ and $C\in\mathbb{H}_{10}$ or vice versa. Assume the former without loss of generality. 

By left division of both $A$ and $B$ by the leading coefficient of $A$ we may achieve $A(v)=v+a$ for some $a\in\mathbb{H}$. Analogously, assume that $C(u)=u+c$ for some $c\in\mathbb{H}$. Performing an appropriate linear change of variables $u$ and $v$ we may achieve $a,c\in\im\H$.
If $a=0$ then $\Phi(u,v)=A(v)^{-1}B(u,v)C(u)^{-1}=D(u,1/v)C(u)^{-1}$ for appropriate $D\in\mathbb{H}_{11}$, hence $\Phi$ is a Darboux cyclide by the previous paragraph. The same holds for $c=0$.

Assume further that $a,c\ne 0$. Then we have
\begin{align*}
\Phi(u,v)-\Phi(u,0)-\Phi(0,v)+\Phi(0,0)&=
A^{-1}(B-Aa^{-1}B(u,0)-B(0,v)c^{-1}C-A\Phi(0,0)C)C^{-1}\\
&=(v+a)^{-1}buv(u+c)^{-1}
\end{align*}
for some $b\in\mathbb{H}$ because the left-hand sides vanishes identically for $u=0$ or $v=0$. If $b=0$ then $\Phi(u,v)=\Phi(u,0)+\Phi(0,v)-\Phi(0,0)$ is a subset of a Euclidean translational surface because by Lemma~\ref{l-axial} the curves $\Phi(u,0)$ and $\Phi(0,v)-\Phi(0,0)$ are circles (not degenerating to points because $\Phi$ is a nondegenerate surface).

Assume further that $b\ne 0$. By the above $(v+a)^{-1}buv(u+c)^{-1}\subset\im\H$ for each $u,v\in\mathbb{R}$. Thus $\mathrm{Re}(v+\bar a)b(u+\bar c)=0$, hence $b\in \im\H$, $b\perp a$, $b\perp c$, and $a\times b\perp c$. Since $a,b,c\ne 0$ this implies that $a\parallel c$.
By Lemma~\ref{l-axial} the curves $u=\mathrm{const}$ and $v=\mathrm{const}$ are circles (or points) whose planes are orthogonal to the vector $a\parallel c$. 
Thus all these circles and hence the surface $\Phi(u,v)$ are contained in one plane.
\end{proof}

\begin{proof}[Proof of Theorem~\ref{prop21} for $\Phi \subset S^3$]
As in the previous proof, we may assume that $A\in\mathbb{H}_{01}$, $C\in\mathbb{H}_{10}$ and $A,C\ne\mathrm{const}$. Since $\Phi \subset S^3$ it follows that
$|A^{-1}BC^{-1}| = 1$ and $|B(u,v)|^2 = |A(v)|^2|C(u)|^2$.
By Splitting Lemma~\ref{l-splitting-basic} there exist $D(v)\in\mathbb{H}_{01}$ and $E(u)\in\mathbb{H}_{10}$
such that  $B(u,v)$ splits:
$B(u,v) = D(v) E(u)$ or $B(u,v) = E(u) D(v)$. Since $|A(v)|\cdot |C(u)|=|B(u,v)|=|D(v)|\cdot |E(u)|$,  without loss of generality we may assume that 
 $|D| = |A|$, $|E| = |C|$.

In the case when $B=DE$ we have $\Phi = A^{-1}DEC^{-1}=\bar A\,|A|^{-2}\,|D|^2\,\bar D^{-1}EC^{-1}=(\bar A \,\bar D^{-1})(EC^{-1})$
is a product of two circles $\bar A(v)\bar D(v)^{-1}$ and $E(u)C(u)^{-1}$ in $S^3$ because $|D| = |A|$, $|E| = |C|$.
Thus $\Phi$ is a subset of a Clifford translational surface.

In the case when $B=ED$ we have
$
\Phi = A^{-1}EDC^{-1} = \bar AE\bar D^{-1}C^{-1} = (\bar AE)(C\bar D)^{-1}
$
because $|D|=|A|$. Since $\bar AE, C\bar D \in \H_{11}$ it follows by Lemma~\ref{prop11} that $\Phi$ is contained in the intersection of $S^3$ with another $3$-dimensional quadric.
\end{proof}

\begin{proof}[Proof of Corollary~\ref{cor-3D}]
Assume that a surface $\Psi\subset\mathbb{R}^3$ is a M\"obius transformation of surface~\eqref{eq-4DBconj}.
Since a M\"obius transformation takes a hyperplane to either a hyperplane or a $3$-dimensional sphere it follows that the latter surface 
is contained either in a hyperplane or a $3$-dimensional sphere.

Perform a M\"obius transformation $q\mapsto aqc+b$  (
a similarity), where $a,b,c\in\H$, taking the obtained hyperplane (respectively, the $3$-dimensional sphere) to $\im\H$ (respectively, to $S^3$). It takes the surface $A^{-1}BC^{-1}$ to the surface  $(Aa^{-1})^{-1}(B+Aa^{-1}bc^{-1}C) (c^{-1}C)^{-1}$
again of form~\eqref{eq-4DBconj}.

By Theorem~\ref{prop21} the resulting surface 
either a Darboux cyclide, or Euclidean translational surface, or intersection of $S^3$ with another quadric, or Clifford translational surface. In the latter two cases perform the inversion with the center at the point $1\in\mathbb{H}$ and the radius $\sqrt{2}$ 
projecting the surface stereographically from $S^3$ to $\im\H$.
This gives either a Darboux cyclide (by \cite[Section~2.2]{PSS11}) or a 
Clifford translational surface.
In all cases the resulting surface in $\im\H$ is a M\"obius transformation of the initial surface $\Psi\subset\mathbb{R}^3$. 
\end{proof}

We conclude the section by an open problem: find a short proof that Euclidean and Clifford translational surfaces are not M\"obius transformations of each other; this is \cite[Theorem 2b]{Lubbes-15}.


\section{Parametrization results}\label{sec:proofs2}

In this section we prove the results on parametrization of surfaces containing two conics or circles through each point: Theorem~\ref{schicho} and Corollary~\ref{haupt}. The proof uses well-known methods and goes along the lines of \cite{schicho:2001}.

We use the following notions.
Let $P^n$ be the $n$-dimensional complex projective space with the homogeneous coordinates $x_0:\dots:x_n$. Throughout we use the standard topology in $P^n$ (not the Zariski one). A \emph{analytic surface} in $P^n$ is the image of an injective complex analytic map from a domain in $\mathbb{C}^2$  into $P^n$ with nondegenerate differential at each point.
An \emph{algebraic subset} $X\subset P^n$ is the solution set of some system of algebraic equations. Algebraic subsets of dimension $1$ and $2$ are called \emph{projective algebraic surfaces} and \emph{algebraic curves}, respectively. Recall that the set of all conics in $P^n$ including the ones degenerating into lines and pairs of lines is naturally identified with an algebraic subset of $P^N$ for some large $N$ (depending on $n$). The latter subset is called \emph{the variety of all conics in $P^n$}. A conic \emph{analytically depending} on a point is a complex analytic map of an analytic surface in $P^n$ into the variety of all conics in $P^n$. An \emph{analytic family} of conics is a complex analytic map $t\mapsto\alpha_t$ of a domain in $\mathbb{C}$ into the variety of all conics in $P^n$. 
If no confusion arises, the image of this map is also called a \emph{family} of conics.

Let us prove several lemmas required for the proof of Theorem~\ref{schicho}. These lemmas are independent in the sense that the proof of each one uses the statements but not the proofs of the other ones. In what follows $\Phi$ is a surface 
satisfying the assumptions of Theorem~\ref{schicho} unless otherwise indicated. Denote by $\alpha_P$ and $\beta_P$ the two conics drawn through a point $P\in\Phi$.

\begin{lem} \label{l-cover} There are
two analytic families of conics $\alpha_t$, $\beta_s$, 
and a domain $\Omega\subset P^n$ such that $\bigcup_{t}\alpha_t\cap\Omega=
\bigcup_{s}\beta_s\cap\Omega=\Phi\cap\Omega\ne\emptyset$, each pair $\alpha_t$, $\beta_s$ intersects transversely at a unique point~$P(s,t)$, and
$\alpha_{P(s,t)}=\alpha_t$, $\beta_{P(s,t)}=\beta_s$.
\end{lem}

\begin{proof}
Take a point $P_0\in\Phi$. Draw the two conics $\alpha_0:=\alpha_{P_0}$ and $\beta_0:=\beta_{P_0}$ in the surface through the point. Through each point $P\in\alpha_0\cap\Phi$ draw another conic $\beta_{P}$ in the surface. We get an analytic family  of conics $\beta_s$. Analogously we get an analytic family  of conics $\alpha_t$.

By the assumptions of Theorem~\ref{schicho} the conics $\alpha_0$ and $\beta_0$  intersect transversely at a unique point. By continuity there is $\epsilon>0$ such that for $|s|,|t|<\epsilon$ the conics $\alpha_t$ and $\beta_s$ intersect transversely at a unique point $P(s,t)$, and $P(s,t)\in\Phi$. Take a sufficiently small neighborhood $\Omega$ of the point $\alpha_0\cap\beta_0$ in $P^n$. Then $\bigcup_{t}\alpha_t\cap\Omega=
\bigcup_{s}\beta_s\cap\Omega=\Phi\cap\Omega\ne\emptyset$.

It remains to show that $\alpha_{P(s,t)}=\alpha_t$ and $\beta_{P(s,t)}=\beta_s$. Indeed, otherwise the image of the analytic map $P\mapsto \alpha_P$ or $P\mapsto \beta_P$ in the variety of all conics in $P^n$ is $2$-dimensional. Then there are infinitely many conics $\alpha$ through each point in an open subset of $\Phi$ such that  $\alpha\cap\Omega\subset\Phi$. This contradicts to one of the assumptions of Theorem~\ref{schicho}. Thus the families $\alpha_t$ and $\beta_s$
are the required.
\end{proof}


\begin{lem} \label{l-algebraic} The surface $\Phi$ is contained in an irreducible algebraic surface $\bar\Phi$. The family of conics $\alpha_t$ is contained in an irreducible algebraic curve in the variety of all conics.
\end{lem}

\begin{proof} 
The conics $\beta_s$ given by Lemma~\ref{l-cover} have at most $4$ common points because they do not all coincide. Thus we may assume that they do not have common points inside the domain $\Omega$ (one can restrict the surface to a smaller domain, if necessary). 
Then by the analyticity each point of $\Omega$ belongs to at most countable number of conics $\beta_s$.
Take a sufficiently small $\epsilon>0$ such that $\alpha_t\cap\beta_s\subset\Omega$ for each $|s|,|t|<\epsilon$.

Consider the set $\gamma$ of all conics in $P^n$ intersecting each conic $\beta_s$, where $|s|<\epsilon$, transversely at a unique point, which belongs to the domain $\Omega$. Let us show that the set $\gamma$ is an open subset of an algebraic subvariety of the variety of all conics in $P^n$. Indeed, the set of all conics intersecting a fixed conic $\beta_s$ is clearly an algebraic subvariety. 
The set of all conics intersecting each conic $\beta_s$, where $|s|<\epsilon$, is the intersection of infinitely many such algebraic subvarieties, and hence also an algebraic subvariety. The set $\gamma$ is open in the latter subvariety because each conic sufficiently close to a conic lying in $\gamma$ can only intersect a conic $\beta_s$, where $|s|<\epsilon$, transversely at a unique point inside the domain $\Omega$ (or not intersect at all).

Let us prove that the dimension of the set $\gamma$ is $1$. The dimension is at least $1$ because $\gamma$ contains all the conics $\alpha_t$, where $|t|<\epsilon$. To estimate the dimension from above,
take an arbitrary conic $\alpha\in\gamma$. Since each point of $\Omega$ belongs to at most countable number of conics $\beta_s$ it follows that the conic $\alpha$ has uncountably many intersection points with $\bigcup_{s}\beta_s\cap\Omega=\Phi\cap\Omega$. Hence $\alpha\cap\Omega\subset\Phi\cap\Omega$, because $\Phi$ is analytic. If the dimension of $\gamma$ were at least $2$ then there would be infinitely many conics $\alpha$ through each point in an open subset of $\Phi$ such that $\alpha\cap\Omega\subset\Phi$. This would contradict to one of the assumptions of Theorem~\ref{schicho}. Thus the dimension of $\gamma$ is exactly $1$.

The  irreducible component of the algebraic closure of the set $\gamma$ containing all the conics $\alpha_t$ is the required algebraic curve in the variety of all conics. (Since the family $\alpha_t$ is analytic, it cannot ``jump'' from one irreducible component to another.) The union of all the conics of the irreducible component (including the ones degenerating to lines and pairs of lines) is the required irreducible algebraic surface $\bar\Phi$. The algebraic surface $\bar\Phi$ contains the analytic surface $\Phi$ because $\bar\Phi$ contains the open subset $\bigcup_{t}\alpha_t\cap\Omega=\Phi\cap\Omega$ of $\Phi$.
\end{proof}

\begin{rem} \label{rem-infinity} 
If we drop the assumption that the number of conic through certain points is finite in Theorem~\ref{schicho} then a similar argument shows that $\Phi$ is contained in an algebraic surface $\bar \Phi$ and $\alpha_t$ is contained in an algebraic subvariety $\gamma$ of the variety of all conics in $P^n$ such that $\bigcup_{\alpha\in\gamma}\alpha=\bar\Phi$.
\end{rem}



Let $X,Y\subset P^n$ be algebraic subsets.
A \emph{rational map} $X\dasharrow Y$ is a map of an open dense subset of $X$ into the set $Y$ given by polynomials in homogeneous coordinates of $P^n$. (Dashes in the notation remind that a rational map may not be defined everywhere in $X$.)
If the restriction of a rational map $f$ to certain open dense subsets of $X$ and $Y$ is bijective, and the inverse map is rational as well, then $f$ is called a \emph{birational map}. A rational map $X\dasharrow P^1$ is called a \emph{rational function}. 

A projective algebraic surface $\Psi$ is \emph{unirationally ruled} (or simply \emph{uniruled}), if for some algebraic curve $\gamma$ there is a rational map $\gamma\times P^1\dasharrow\Psi$ with dense image. A surface $\Psi$ is \emph{birationally ruled} (or simply \emph{ruled}), if there is a birational map $\gamma\times P^1\dasharrow\Psi$. A curve $\gamma$ is \emph{rational}, if there is
a birational map $P^1\dasharrow\gamma$.
A surface $\Psi$ is \emph{rational}, if there is a birational map $P^1\times P^1\dasharrow\Psi$.


\begin{lemma} \label{l-uniruled} The surface $\bar\Phi$ is unirationally ruled.
\end{lemma}

\begin{proof} 
Let $\bar\gamma$ be the irreducible curve in the variety of conics given by Lemma~\ref{l-algebraic}.
Consider the algebraic set
$\bar\Psi := \{(P, \alpha)\in \bar\Phi\times\bar\gamma : P\in\alpha\}.
$
The second projection $\bar\Psi\to \bar\gamma$ is a rational map such that a generic fiber is a conic (and hence a rational curve). By the Noether--Enriques theorem \cite[Theorem~III.4]{Beauville}  $\bar\Psi$ is birationally ruled. In particular there is a rational map $\bar\gamma\times P^1\dasharrow \bar\Psi$ with dense image. Compose the map with the first projection $\bar\Psi\to\bar\Phi$, which is surjective because the surfaces $\bar\Psi$ and $\bar\Phi$ are compact and the image contains the open subset $\bigcup_t\alpha_t\cap\Omega=\Phi\cap\Omega$ by Lemma~\ref{l-cover}. We get a rational map $\bar\gamma\times P^1\dasharrow \bar\Phi$ with dense image, i.e., $\bar\Phi$ is unirationally ruled.
\end{proof}

The following folklore lemma is the most technical part of our proof.

\begin{lemma} \label{l-ruled} Each unirationally ruled surface is birationally ruled.
\end{lemma}

\begin{proof} 
Let $\bar\Phi$ be a unirationally ruled surface and $\gamma\times P^1\dasharrow\bar\Phi$ be a rational map with dense image.
By the Hironaka theorem 
(or by an earlier Zariski theorem sufficient in our situation) the surface $\bar\Phi$ has a \emph{desingularization} $d:\tilde\Phi\to\bar\Phi$, i.e., a proper birational map from a smooth projective algebraic surface $\tilde\Phi$ to the surface $\bar\Phi$.
Let $\bar\Phi\dasharrow\tilde\Phi$ be the inverse rational map of the desingularization.

Consider the rational map $\gamma\times P^1\dasharrow \bar\Phi\dasharrow \tilde\Phi$.
By the theorem on eliminating indeterminacy \cite[Theorem II.7]{Beauville} this rational map equals to a composition $\gamma\times P^1\dasharrow \tilde\Psi\to \tilde\Phi$, where  $\tilde\Psi$ is a smooth projective algebraic surface, the first map is birational, and the second map is rational and defined everywhere. Since $\gamma\times P^1\dasharrow \bar\Phi$ has dense image and the surfaces are compact it follows that $\gamma\times P^1\dasharrow \tilde\Psi$ has dense image and $\tilde\Psi\to \tilde\Phi$ is surjective.
In particular, $\tilde\Psi$ is birationally ruled.

By the Enriques theorem \cite[Theorem~VI.17 and Proposition~III.21]{Beauville}, 
a smooth projective algebraic surface is birationally ruled if and only if for each $k>0$ the $k$-th tensor power of the exterior square of the cotangent bundle has no sections except identical zero (in other terminology, \emph{the surface has Kodaira dimension} $-\infty$, or \emph{all plurigeni vanish}).
Assume, to the contrary, that $\tilde\Phi$ has such a section (\emph{pluricanonical section}). 
Then the pullback under the surjective rational map $\tilde\Psi\to\tilde\Phi$ 
is such a section for the birationally ruled surface $\tilde\Psi$, a contradiction. Thus $\tilde\Phi$, and hence $\bar\Phi$, is birationally ruled.
\end{proof}

\begin{lemma} \label{l-rational} The surface $\bar\Phi$ is rational.
\end{lemma}

\begin{proof}
By Lemmas~\ref{l-uniruled} and~\ref{l-ruled} the surface $\bar\Phi$ is birationally ruled. Thus there is a birational map $\bar\Phi\dasharrow\gamma\times P^1$. Consider the two conics through a generic point of the surface $\bar\Phi$. Their images are two distinct rational curves through a point of $\gamma\times P^1$. Since there is only one $P^1$-fiber through each point, at least one of the rational curves is nonconstantly projected to the curve $\gamma$. By the L\"uroth theorem \cite[Theorem~V.4]{Beauville} the curve $\gamma$ must be rational, and hence $\bar\Phi$  is rational.
\end{proof}

\begin{rem} We conjecture that the following generalization of Lemma~\ref{l-rational} is true: \emph{an algebraic surface containing two rational curves through almost each point is rational}. See~\cite[Definition~IV.3.2, Theorems~ IV.5.4, IV.3.10.3, 
Corollary~IV.5.2.1, Exercise~IV.3.12.2]{Kollar} for a sketch of the proof.
\end{rem}

\begin{lem}\label{l-linear-pencil} Each of the families $\alpha_t$ and $\beta_s$ 
consists of level sets of some rational function~$\bar\Phi\dasharrow P^1$.
\end{lem}

\begin{proof} 
We use the following notions; see \cite[\S III.1]{Shafarevich} for details. In order to apply intersection theory, take a \emph{desingularization} $d:\tilde\Phi\to\bar\Phi$; see the first paragraph of the proof of Lemma~\ref{l-ruled}.
A \emph{divisor} on $\tilde\Phi$ is a formal linear combination of irreducible algebraic curves on $\tilde\Phi$ with integer coefficients. Closure of the preimage of an algebraic hypersurface under a rational map from $\tilde\Phi$ to a projective space can be considered as a divisor, once the preimage is one-dimensional and one counts the irreducible components with their multiplicities.
The collection of preimages of all the hyperplanes under a rational map is called a 
\emph{linear family} of divisors.
In particular, a one-dimensional \emph{linear} 
\emph{family} of divisors is the collection of level sets of a rational function. 
Two divisors are \emph{linear equivalent}, 
if they are contained in a one-dimensional linear 
family.
The \emph{intersection} $D_1\cap D_2$ of two divisors $D_1$ and $D_2$ on $\tilde\Phi$ is the number of their intersection points counted with multiplicities (once the number of intersection points is finite).
Two divisors are \emph{numerically equivalent}, if their intersection with
each algebraic curve on $\tilde\Phi$ are equal, once the number of intersection points is finite. The \emph{degree} of a divisor is the sum of the degrees of the irreducible components counted with multiplicities. The degree of a divisor equals the intersection of the divisor with a generic hyperplane section of $\tilde \Phi$.

Assume to the contrary that one of the analytic families $\alpha_t$ and $\beta_s$, say, the first one, is not linear.
By the Hironaka theorem 
the inverse map $d^{-1}\colon \bar\Phi\dasharrow\tilde \Phi$ is defined everywhere except a finite set. Thus the pullback $d^{-1}\alpha_t$ is an analytic family of algebraic curves. We use the notation $d^{-1}\alpha_t$ for the family of (closed) algebraic curves (not to be confused with the set of preimages $d^{-1}(\alpha_t)$ being algebraic curves possibly with a finite number of points removed).
For an algebraic curve $\beta\subset\tilde\Phi$ distinct from
each ${d^{-1}\alpha_t}$
the intersection $\beta\cap {d^{-1}\alpha_t}$ continuously depends on $t$, and hence is constant. Thus each two curves of the family ${d^{-1}\alpha_t}$ are numerically equivalent.
On a smooth rational surface, numerical equivalence implies linear equivalence 
\cite{Zariski}.
Since $\alpha_t$ is nonlinear it follows that ${d^{-1}\alpha_t}$ is nonlinear and hence must be contained in a linear family of divisors in $\tilde\Phi$ of dimension at least $2$.

The image of the latter family under the projection $d\colon\tilde\Phi\to\bar\Phi$ is at least two-dimensional algebraic family $\gamma$ of divisors on $\bar\Phi$. The divisors of the family $\gamma$ are linear combinations of curves with positive coefficients, because they arise from a linear family (and negative numbers cannot occur as multiplicities of preimages). All the divisors of the family $\gamma$ have the same degree because their pullbacks are numerically equivalent. (Indeed, for two divisors $D_1,D_2\in\gamma$, and a general position hyperplane $H\subset P^n$ we have $\deg D_1=H\cap D_1=d^{-1}H\cap d^{-1}D_1=d^{-1}H\cap d^{-1}D_2=H\cap D_2=\deg D_2$.) 
Since the curves $\alpha_t$ are conics it follows that the degrees of all these divisors are $2$.
Therefore the divisors of the 
family $\gamma$ are either conics or pairs of lines or lines of multiplicity $2$.  
This is possible only if there are infinitely many conic sections or lines through each point in an open subset of the surface, which contradicts to the assumptions of Theorem~\ref{schicho}. Thus both families $\alpha_t$ and $\beta_s$ must be linear.
\end{proof}

\begin{lemma}\label{l-parametrization}
There is a rational map $\bar\Phi\dasharrow P^1\times P^1$
taking the families $\alpha_t$ and $\beta_s$ to the sets of curves $P^1\times t$ and $s\times P^1$ respectively, such that the restriction of the map to some dense sets is bijective.
\end{lemma}

\begin{proof}
Consider the pair of rational functions given by Lemma~\ref{l-linear-pencil} whose level sets are the two families  of conics $\alpha_t$ and $\beta_s$. The pair of rational functions defines a rational map $\bar\Phi\dasharrow P^1\times P^1$. Since for sufficiently small $|s|, |t|$ each pair $\alpha_t$ and $\beta_s$  has an intersection point it follows that the image of this rational map contains a neighborhood of $(0,0)\in P^1\times P^1$, and thus is dense. Since the intersection point is unique it follows that the point $(s,t)\in P^1\times P^1$ has exactly one preimage for sufficiently small $|s|, |t|$, and hence for almost all $s,t$. Thus  the restriction of the rational map $\bar\Phi\dasharrow P^1\times P^1$ to appropriate dense subsets  is bijective.
\end{proof}

Now we apply the following well-known result for which we could not find any direct reference. (We have found several more general results in the literature but each time the proof that the inverse map is rational, the only assertion we need, was omitted.)

\begin{lemma}\label{l-postulate} If a rational map between open dense subsets of rational surfaces is bijective then the inverse map is rational as well, and hence birational.
\end{lemma}

\begin{proof} (S. Orevkov, private communication) It suffices to prove the lemma for a map $f$ between open dense subsets of $P^1\times P^1$. For a generic $s\in P^1$
the preimage $f^{-1}(s\times P^1)$ is an open dense subset of the algebraic curve $\beta$ defined by the algebraic equation $\mathrm{pr}_1f(u,v)=s$ in the variables $u,v\in P^1$, where $\mathrm{pr}_1\colon P^1\times P^1\to P^1$ is the first projection. 

Restrict the map $f$ to the preimage. We get a bijective rational map between open dense subsets of the curves $\beta$ and $s\times P^1$. Clearly, it extends to a rational homeomorphism $\beta\to s\times P^1$. Then by the classification of algebraic curves the curve $\beta$ is rational. Identify $\beta$ and $s\times P^1$ with $P^1$. Consider the graph of the rational map $\beta\to s\times P^1$ as a subset of $P^1\times P^1$. By elimination of variables it follows that the graph is the zero set of some polynomial $P\in \mathbb{C}[u,v]$. Since the map $\beta\to s\times P^1$ is bijective it follows that the polynomial $P$ has degree $1$ in each variable and hence the inverse map $s\times P^1\to\beta$ is also rational. 

This implies that the inverse map $f^{-1}(s,t)$ is rational in the variable $t$ for fixed generic $s$. Analogously, $f^{-1}(s,t)$ is rational in $s$ for fixed generic $t$. 
Thus $f^{-1}(s,t)$ is rational. 
\end{proof}

\begin{lemma}\label{l-final} Assume that a birational map $P^1\times P^1\dasharrow \bar\Phi$ takes the sets of curves $P^1\times t$ and $s\times P^1$ to conics or lines. Write the birational map as
$(u,v)\mapsto X_0(u,v):\dots :X_n(u,v)$ for some coprime $X_0,\dots,X_n\in\mathbb{C}[u,v]$.
Then $X_0,\dots,X_n\in\mathbb{C}_{22}$.
\end{lemma}

\begin{proof}
For a birational map between surfaces, by elimination of indeterminacy \cite[Theorem~II.7]{Beauville} there is always an algebraic curve $\sigma$ such that  outside the curve $\sigma$ the map is injective and has nondegenerate differential.
Fix a generic value of $u$. Denote $X_k(v):=X_k(u,v)$. The curve $X_0(v):\dots :X_n(v)$ is a conic or a line.
Cut it by a generic hyperplane $\lambda_0 x_0+\dots \lambda_nx_n=0$ in $P^n$. The intersection consists of at most $2$ points. By general position they are not contained in the image of $\sigma$.
Since the $X_0(u,v),\dots,X_n(u,v)$ are coprime it follows that
$X_0(v),\dots,X_n(v)$ have no common roots.
Since the birational map is injective outside $\sigma$ it follows that the equation $\lambda_0 X_0(v)+\dots +\lambda_nX_n(v)=0$ has at most $2$ solutions. These solutions have multiplicity $1$ because the birational map has nondegenerate differential outside $\sigma$.
This implies that the polynomials $X_1(v),\dots,X_n(v)$ have degree at most $2$. Analogously, $X_1(u,v),\dots,X_n(u,v)$ have degree at most $2$ in $u$, and the lemma follows.
\end{proof}

\begin{proof}[Proof of Theorem~\ref{schicho}] It follows directly by Lemmas~\ref{l-cover}, \ref{l-algebraic}, \ref{l-parametrization}, \ref{l-postulate}, \ref{l-final}.
\end{proof}

\begin{rem} Theorem~\ref{schicho} remains true (with almost the same proof) without the assumption that the number of conic sections through certain points is finite except that then one cannot conclude that the two drawn conic sections $\alpha_P$ and $\beta_P$ are necessarily the curves $u=\mathrm{const}$ and $v=\mathrm{const}$.
\end{rem}

For the proof of Corollary~\ref{haupt} we need the following lemmas. In the rest of this section $\Phi\subset S^{n-1}\subset\mathbb{R}^n$ is a surface satisfying the assumptions of Corollary~\ref{haupt}.

\begin{lem} \label{l-reduction} The surface $\Phi$ (possibly besides a one-dimensional subset) has parametrization~\eqref{eq-circular},
where $X_{0},\dots,X_{n}\in\mathbb{C}_{22}$
satisfy equation~\eqref{eq-sphere},
and $(u,v)$ runs through some (not open) subset of~$\mathbb{C}^2$.
\end{lem}

\begin{proof} 
 Since the circular arcs through each point of $\Phi$ are noncospheric, their respective circles are transversal and intersect at a unique point.
Extend $\Phi\subset S^{n-1}$ analytically to a complex analytic surface $\bar\Phi$ in a sufficiently small neighborhood of $\Phi$ in $P^n$ modulo the boundary (so that the boundaries of $\Phi$ and $\bar\Phi$ are contained in the boundary of the neighbourhood).
Extend the two real analytic families of circular arcs in $\Phi$ to complex analytic families of (complex) conics in $P^n$.
By analyticity $\bar\Phi$ satisfies the assumptions of Theorem~\ref{schicho}.
By Theorem~\ref{schicho} the surface $\bar\Phi$ (possibly besides a one-dimensional subset) has parametrization~\eqref{eq-circular} by polynomials $X_0,\dots,X_n\in\mathbb{C}_{22}$ such that the circular arcs have the form $u=\mathrm{const}$ and $v=\mathrm{const}$. (However, the converse is not true: most of the curves $u=\mathrm{const}$ and $v=\mathrm{const}$ are not circular arcs but parts of complex conics in $\bar\Phi$.) Since the surface $\bar\Phi$ is contained in the complex quadric extending the sphere $S^{n-1}$, it follows that the polynomials satisfy equation~\eqref{eq-sphere}.
\end{proof}

Let us reparametrize the surface to make the polynomials $X_0,\dots,X_n$ real.

\begin{lem} \label{l-make-real} The surface $\Phi$ (possibly besides a one-dimensional subset) has a parametrization~\eqref{eq-circular},
where $X_{0},\dots,X_{n}\in\mathbb{C}_{22}$
satisfy equation~\eqref{eq-sphere},
and $(u,v)$ runs through certain open subset of~$\mathbb{R}^2$.
\end{lem}

\begin{proof} Start with the parametrization given by Lemma~\ref{l-reduction}. Draw two circular arcs of the form $u=\mathrm{const}$ and $v=\mathrm{const}$ through a point of the surface $\Phi$. Through another pair of points of the first circular arc, draw two more circular arcs of the form $v=\mathrm{const}$.
Perform a complex fractional-linear transformation of the parameter $v$ so that the second, the third, and the fourth circular arcs obtain the form $v=0,\pm1$, respectively (we consider the part of the surface where the denominator of the transformation does not vanish).
Perform an analogous transformation of the parameter $u$ so that $u=0,\pm1$ become circular arcs intersecting the circular arc $v=0$.
After performing the transformations and clearing denominators we get a parametrization $\Phi(u,v)=X_{0}(u,v):\dots :X_{n}(u,v)$ of the surface $\Phi\subset S^{n-1}$, where still $X_{0},\dots,X_{n}\in\mathbb{C}_{22}$ and $(u,v)$ runs through a subset $\Psi\subset\mathbb{C}^2$.

Let us prove that actually 
$\Psi\subset\mathbb{R}^2$. Take $(\hat u,\hat v)\in\Psi$ sufficiently close to $(0,0)$. Then $\Phi(\hat u,\hat v)$ is a point of the surface $\Phi$ sufficiently close to $\Phi(0,0)$. Draw the two circular arcs $u=\mathrm{const}$ and $v=\mathrm{const}$ through the point $\Phi(\hat u,\hat v)$.  By continuity it follows that the circular arc $v=\mathrm{const}$ intersects the circular arc $u=0$ in $\Phi$. The intersection point can only be $\Phi(0,\hat v)$. In particular, we get $(0, \hat v)\in\Psi$.
In a quadratically parametrized conic, the cross-ratio of any four points equals the cross-ratio of their parameters. Since three (real) points $v=0,\pm 1$ of the circular arc $u=0$ have real $v$-parameters it follows that all (but one) points of the circular arc have real $v$-parameters. In particular, $\hat v\in\mathbb{R}$. Analogously, $\hat u\in\mathbb{R}$. We have proved that all $(\hat u,\hat v)\in\Psi$ sufficiently close to the origin are real. By the analyticity $\Psi$ is an open subset of $\mathbb{R}^2$.
\end{proof}

\begin{lem} \label{l-real} Let $X_0,\dots,X_n\in \mathbb{C}[u,v]$. Assume that for all the points $(u,v)$ from some open subset of $\mathbb{R}^2$ the point $X_0(u,v):\dots:X_n(u,v)$ is real. Then  $X_0=X_0'Y,
\dots,X_n=X_n'Y$ for some real $X_0',\dots,X_n'\in \mathbb{R}[u,v]$ and complex $Y\in\mathbb{C}[u,v]$.			
\end{lem}

\begin{proof}[Proof of Lemma~\ref{l-real}]
Without loss of generality assume that $X_0$ is not identically zero.
Take $1\le l\le n$ such that $X_l$ is not identically zero and take generic real $u,v$ from the open set in question. By the assumption of the lemma $X_0(u,v):\dots:X_n(u,v)$ is real. Hence $X_0(u,v)/X_l(u,v)$ is real, therefore $X_0(u,v)/X_l(u,v)=\bar X_0(u,v)/\bar X_l(u,v)$. Since this holds for generic real $u,v$ it follows that $X_0/X_l=\bar X_0/\bar X_l$ as polynomials.
Take decompositions
\begin{align*}
X_0&=\lambda_0 Y_{1}^{p_{01}}\bar Y_{1}^{q_{01}}Z_{1}^{r_{01}}\dots Y_{m}^{p_{0m}}\bar Y_{m}^{q_{0m}}Z_{m}^{r_{0m}},\\
\dots\\
X_n&=\lambda_n Y_{1}^{p_{n1}}\bar Y_{1}^{q_{n1}}Z_{1}^{r_{n1}}\dots Y_{m}^{p_{nm}}\bar Y_{m}^{q_{nm}}Z_{m}^{r_{nm}},
\end{align*}
  into coprime reduced irreducible factors $Y_{1},\dots,Y_{m}\in\mathbb{C}[u,v]-\mathbb{R}[u,v]$, $Z_{1},\dots,Z_{m}\in\mathbb{R}[u,v]$ with the powers $p_{ij},q_{ij},r_{ij}\ge 0$,  and constant factors $\lambda_0,\dots,\lambda_n\in\mathbb{C}$, $\lambda_0\ne0$.

Since $\mathbb{C}[u,v]$ is a unique factorization domain, the relation $X_0/X_l=\bar X_0/\bar X_l$ implies that $\lambda_l\bar\lambda_0\in\mathbb{R}$ and $p_{0k}-q_{0k}=\dots=p_{nk}-q_{nk}$ for each $k=1,\dots,m$. Without loss of generality assume that $p_{0k}-q_{0k}\ge 0$ for each $k=1,\dots,m$ (otherwise replace $Y_k$ by $\bar Y_k$ and vice versa). It remains to set
\begin{align*}
 X'_l&:=\lambda_l\bar\lambda_0\left(Y_{1}\bar Y_{1}\right)^{q_{l1}}Z_{1}^{r_{l1}}\dots \left(Y_{m}\bar Y_{m}\right)^{q_{lm}}Z_{m}^{r_{lm}};\\
 Y&:=\bar\lambda_0^{-1}Y_{1}^{p_{01}-q_{01}}\dots Y_{m}^{p_{0m}-q_{0m}}.
\\[-1.4cm]
\end{align*}
\end{proof}

\begin{proof}[Proof of Corollary~\ref{haupt}]
For a surface $\Phi$ in $S^{n-1}$ the corollary follows from Lemmas~\ref{l-reduction}--\ref{l-real}.
Now consider a surface $\Phi$ in $\mathbb{R}^{n-1}$. Project it stereographically to $S^{n-1}$. The resulting surface has a parametrization as granted by the theorem for $S^{n-1}$. Thus the initial surface has the parametrization
$\Phi=X_{0}-X_{n}:X_{1}:\dots:X_{n-1}$ with
$X_{1}^2+\dots+X_{n-1}^2=(X_{0}-X_{n})(X_{0}+X_{n})$, as required.
\end{proof}

The following result is useful for applications of Corollary~\ref{haupt}.

\begin{lem} If through each point of an analytic surface in $\mathbb{R}^3$ one can draw infinitely many circular arcs fully contained in the surface then the surface is a subset of a sphere or a plane.
\end{lem}

\begin{proof} Perform inverse stereographic projection of the surface to $S^3$. By Remark~\ref{rem-infinity} the resulting surface is covered by at least 2-dimensional algebraic family of conics and hence circles. The circles of the family passing through a particular point must cover an open subset of the surface. Project the surface back to $\mathbb{R}^3$ from this point.
We get a surface containing a line segment and infinitely many circular arcs through almost each point. By \cite[Theorem~1]{NS11} the resulting surface is a subset of a quadric or a plane. Since it is covered by a 2-dimensional family of circles, by the classification of quadrics the resulting surface, and hence initial one, is a subset of a sphere or a plane.
\end{proof}

We conclude the section by an open problem: does the lemma remain true in dimension $n>3$?

\section{Factorization results}\label{sec:final}

In this section we prove the results concerning factorization of quaternionic polynomials: Splitting Lemma~\ref{l-splitting-basic}, Corollary~\ref{cor-splitting}, and give some examples and final remarks.


\begin{proof}[Proof of Splitting Lemma~\ref{l-splitting-basic}]
%
%
%
Write $Q(u,v)=:Q_0(u)+vQ_1(u)=:Q_{00}+Q_{10}u+Q_{01}v+Q_{11}uv$ with $Q_0,Q_1\in\mathbb{H}_{10}$, $Q_{11}\ne 0$. Denote $q:=-Q_{11}^{-1}Q_{10}\in\mathbb{H}$.

Consider the polynomial $|Q|^2(u,q)$ obtained by substitution of the quaternion $q$ into the \emph{real} polynomial $|Q|^2(u,v)$. 
On one hand, $|Q|^2(u,q)=P(u)R(q)$ is divisible by $P(u)$ of degree $2$. On the other hand,  
$|Q|^2(u,q)=q(qQ_1+Q_0)\bar Q_1+(qQ_1+Q_0)\bar Q_0$ has degree at most $1$ because $qQ_1+Q_0=Q_{00}-Q_{01}Q_{11}^{-1}Q_{01}=:p$ is a constant. Thus $|Q|^2(u,q)=0$ identically, i.e., $qp\bar Q_1+p\bar Q_0=0$.

Now for $p=0$ we get $Q_0=-qQ_1$, hence $Q=Q_0+vQ_1=(v-q)Q_1$ is reducible. 
For $p\ne 0$ we get $Q_0=-Q_1\bar p\,\bar q\,\bar p^{-1}$, hence $Q=Q_1(v-\bar p\,\bar q\,\bar p^{-1})$ is again reducible.
\end{proof}

\begin{proof}[Proof of Corollary~\ref{cor-splitting}] It follows directly from Splitting Lemma~\ref{l-splitting-basic} applied to $Q:=X_1+iX_2+jX_3+kX_4$ with $|Q|^2=(X_6+X_5)(X_6-X_5)$.
\end{proof}


The following example shows that Splitting Lemma~\ref{l-splitting-basic} does not hold for degree $2$ polynomials.

\begin{exmp} \label{ex-Beauregard} (Beauregard \cite{Beauregard-93}).
The polynomial
$Q=u^2v^2-1 + (u^2- v^2)i + 2uvj$ is irreducible in $\mathbb{H}[u,v]$ but $|Q|^2=(u^4+1)(v^4+1)$ is reducible in $\mathbb{R}[u,v]$. In particular,
$$
Q\bar Q=(u-e^{i\pi/4})(u-e^{3i\pi/4})(u-e^{5i\pi/4})(u-e^{7i\pi/4})(v-e^{i\pi/4})(v-e^{3i\pi/4})(v-e^{5i\pi/4})(v-e^{7i\pi/4})
$$
are two decompositions in $\mathbb{H}[u,v]$ with different number of irreducible factors.
\end{exmp}

\begin{proof} Let us prove that $Q$ is irreducible over $\mathbb{H}[u,v]$ (in \cite{Beauregard-93} the irreducibility only  over \emph{rational} quaternions is proved). Assume that $Q=PR$, where $P,R\in\mathbb{H}_{**}$ are not constant.

First consider the case when one of these polynomials, say, $P$, does not depend on one of variables, say, $v$. Then write $Q=(u^2-i)v^2+(2ju)v+(iu^2-1)$. Both $u^2-i$ and $2ju$ must be left-divisible by $P\ne\mathrm{const}$, a contradiction by taking $u=0$.

It remains to consider the case when $P,R\in\mathbb{H}_{11}$. Since $$|P|^2|R|^2=|Q|^2=(u^2+\sqrt{2}u+1)(u^2-\sqrt{2}u+1)(v^2+\sqrt{2}v+1)(v^2-\sqrt{2}v+1)$$ and $\mathbb{R}[u,v]$ is a unique factorization domain it follows that $|P|^2$ is product of two quadratic factors. 
By Splitting Lemma~\ref{l-splitting-basic} the polynomial $P$ is reducible in $\mathbb{H}[u,v]$. The left factor in a decomposition of $P\in\mathbb{H}_{11}$ is a left divisor of $Q$ and does not depend on one of the variables, which leads to the first case already considered above.
In both cases we get a contradiction which proves that $Q$ is irreducible.
\end{proof}

\begin{conj} \label{conj-1-var} Polynomials $X_{1},\dots,X_6\in\mathbb{R}[u]$ satisfy $X_{1}^2+\dots+X_5^2=X_6^2$ if and only if for some
$A,B\in\mathbb{H}[u]$, $D\in\mathbb{R}[u]$ we have
\begin{equation*}\label{eq-one-variable}
X_1+iX_2+jX_3+kX_4=2ABD,\quad
X_5=(|B|^2-|A|^2)D,\quad
X_6=(|B|^2+|A|^2)D.
\end{equation*}
\end{conj}

It is interesting, if Splitting Lemma~\ref{l-splitting-basic} remains true for polynomials in more than $2$ variables.

It is interesting to obtain \emph{octonion} counterparts of the results of this section. In particular, this could help to find all surfaces in $\mathbb{R}^n$ containing two circles through each point for $n>4$.

Let us give some final remarks.
Simple formula~\eqref{eq-4DBconj} produces surfaces in $\mathbb{R}^4$ containing two circles through each point. However, it does not give \emph{all} such surfaces in $\mathbb{R}^4$ and is not convenient to produce such surfaces in $\mathbb{R}^3$. Thus we suggest the following 
alternative approach to surface construction.

The \emph{(top-left) quasideterminant} of a $3\times 3$ matrix $M=(M_{ij})$ with the entries from $\mathbb{H}[u,v]$ is the following expression \cite{Gelfand-etal-05}:
\begin{multline*}
|M|_{11}:=M_{11}-M_{12}(M_{22}-M_{23}M^{-1}_{33} M_{32})^{-1}M_{21}- M_{12}(M_{32}-M_{33}M^{-1}_{23}M_{22})^{-1}M_{31}-\\
M_{13}(M_{23}-M_{22}M^{-1}_{32} M_{33})^{-1}M_{21}-
M_{13}(M_{33}-M_{32}M^{-1}_{22} M_{23})^{-1}M_{31}.
\end{multline*}
If $M_{22}$ is linear in $u$, $M_{33}$ is linear in $v$, and all the other entries $M_{ij}$ are generic constants then by the homological relations \cite[Theorem~1.4.2(i)]{Gelfand-etal-05} the quasideterminant is fraction-linear in each of the variables $u$ and $v$. Thus $|M|_{11}(u,v)$ is a surface in $\mathbb{R}^4$ containing $2$ 
circles through each point (by Lemma~\ref{l-axial} above). If the matrix $M$ is skew-hermitian then the surface actually lies in $\mathbb{R}^3$.

Similarly, if $M_{11}=M_{12}=M_{21}=0$, $M_{31}$ is a constant, $M_{13}$, $M_{32}$, $M_{33}$ are linear in $u$, $M_{22}$, $M_{23}$ are linear in $v$ then $|M|_{11}(u,v)$ is again a surface in $\mathbb{R}^4$ containing $2$ 
circles through each point.

It is interesting, if there is a natural class of matrices with the entries from $\mathbb{H}[u,v]$ such that \emph{all} surfaces containing $2$ noncospheric circles through each point arise as the quasideterminants. 

\subsection*{Acknowledgements}

The authors are grateful to N. Lubbes for the movie in Figure~\ref{movie} and many useful remarks, to L. Shi for the photo in Figure~\ref{movie}, to A. Pakharev for pointing out that numerous surfaces we tried to invent 
have form~\eqref{eq-4DBconj}, and to S. Galkin, S. Ivanov, W. K\"uhnel, N. Lubbes, S. Orevkov, R. Pignatelli, F. Polizzi, H. Pottmann, G. Robinson, J. Schicho, K. Shramov, V. Timorin, M. Verbitsky, S. Zub\.e for useful discussions. The first author is grateful to King Abdullah University of Science and Technology for hosting him during the start of the work over the paper.



\small
\noindent
\textsc{Mikhail Skopenkov\\
Faculty of Mathematics, National Research University Higher School of Economics, and\\
Institute for Information Transmission Problems, Russian Academy of Sciences} 
\\
\texttt{skopenkov@rambler.ru} \quad \url{http://skopenkov.ru}

\bigskip

\noindent
\textsc{Rimvydas Krasauskas\\
Faculty of Mathematics and Informatics, Vilnius University}\\
\texttt{rimvydas.krasauskas@mif.vu.lt}


\end{document}